\documentclass[runningheads,envcountsame]{llncs}

\usepackage{latexsym,amsmath,amssymb}
\usepackage{graphicx,epsfig}
\usepackage{url}
\usepackage{multirow}
\usepackage[nocompress]{cite}

\newcommand{\Suff}{\textit{Suff}}
\newcommand{\Pref}{\textit{Pref}}
\newcommand{\Fact}{\textit{Fact}}
\renewcommand{\epsilon}{\varepsilon}

\newcommand{\St}{\textit{St}}

\newcommand{\oc}{\textit{oc}}
\newcommand{\wt}{\widetilde}

\begin{document}

\sloppy

\title{Open and Closed Prefixes of Sturmian Words}

\author{Alessandro De Luca \inst{1} \and Gabriele Fici \inst{2}}

\institute{DIETI, Universit\`a di Napoli Federico II, Italy\\ \email{alessandro.deluca@unina.it} \and Dipartimento di Matematica e Informatica, Universit\`a di Palermo, Italy\\ \email{gabriele.fici@unipa.it}}

\maketitle

\begin{abstract}
  A word is closed if it contains a proper factor that occurs both as a prefix and as a suffix but does not have internal occurrences, otherwise it is open. We deal with the sequence of open and closed prefixes of Sturmian words and prove that this sequence characterizes every finite or infinite Sturmian word up to isomorphisms of the alphabet. We then characterize the combinatorial structure of the sequence of open and closed prefixes of standard Sturmian words. We prove that every standard Sturmian word, after swapping its first letter, can be written as an infinite product of squares of reversed standard words.
\end{abstract}


\section{Introduction}
In a recent paper with M.~Bucci~\cite{BuDelFi13}, the authors dealt with trapezoidal words, also with respect to the property of being \emph{closed} (also known as periodic-like~\cite{CaDel01a}) or open. 
Factors of Sturmian words are the most notable example of trapezoidal words, and in fact the last section of~\cite{BuDelFi13}
showed the sequence of open and closed prefixes of the Fibonacci word, a famous characteristic Sturmian word.

In this paper we build upon such results, investigating the sequence of open and closed prefixes of Sturmian words in general, and in particular in the standard case. More precisely, we prove that the sequence $oc(w)$ of open and closed prefixes of a word $w$ (i.e., the sequence whose $n$-th element is $1$ if the prefix of length $n$ of $w$ is closed, or $0$ if it is open) characterizes every (finite or infinite) Sturmian word, up to isomorphisms of the alphabet. 

In~\cite{BuDelFi13}, we investigated the structure of the sequence $oc(F)$ of the Fibonacci word. We proved that the lengths of the runs (maximal subsequences of consecutive equal elements) in $oc(F)$ form the doubled Fibonacci sequence. We prove in this paper that this doubling property holds for every standard Sturmian word, and describe the sequence $oc(w)$ of a standard Sturmian word $w$ in terms of the \emph{semicentral} prefixes of $w$, which are the prefixes of the form $u_{n}xyu_{n}$, where $x,y$ are letters and $u_{n}xy$ is an element of the standard sequence of $w$. As a consequence, we show that the word $ba^{-1}w$, obtained from a standard Sturmian word $w$ starting with letter $a$ by swapping its first letter, can be written as the infinite product of the words $(u_{n}^{-1}u_{n+1})^{2}$, $n\ge 0$. Since the words $u_{n}^{-1}u_{n+1}$ are reversals of standard words, this induces an infinite factorization of $ba^{-1}w$ in squares of reversed standard words. 

Finally, we show how the sequence of open and closed prefixes of a standard Sturmian word of slope $\alpha$ is related to the continued fraction expansion of $\alpha$.

\section{Open and Closed Words}

Let us begin with some notation and basic definitions; for those not included below, we refer the reader to~\cite{BuDelFi13} and~\cite{LothaireAlg}.

Let $\Sigma=\{a,b\}$ be a $2$-letter alphabet. Let $\Sigma^{*}$ and $\widehat{\Sigma}^{*}$ stand respectively for the free monoid and the free group generated by $\Sigma$. Their elements are called \emph{words} over $\Sigma$. The \emph{length} of a word $w$ is denoted by $|w|$. The \emph{empty word}, denoted by $\epsilon$, is the unique word of length zero and is the neutral element of $\Sigma^{*}$ and $\widehat{\Sigma}^{*}$.
 
A \emph{prefix} (resp.~a \emph{suffix}) of a word $w$ is any word $u$ such that $w=uz$ (resp.~$w=zu$) for some word $z$. A \emph{factor} of $w$ is a prefix of a suffix (or, equivalently, a suffix of a prefix) of $w$.  An \emph{occurrence} of a factor $u$ in $w$ is a factorization $w=vuz$. An occurrence of $u$ is \emph{internal} if both $v$ and $z$ are non-empty. The set of prefixes, suffixes and factors of the word $w$ are denoted  by $\Pref(w)$, $\Suff(w)$ and $\Fact(w)$, respectively.
From the definitions, we have that $\epsilon$ is a prefix, a suffix and a factor of any word. 
A \emph{border} of a word $w$ is any word in $\Pref(w)\cap \Suff(w)$ different from $w$.

A factor $v$ of a word $w$ is \emph{left special in $w$} (resp.~\emph{right special in $w$}) if $av$ and $bv$ are factors of $w$ (resp.~$va$ and $vb$ are factors of $w$). A \emph{bispecial factor} of $w$ is a factor that is both left and right special.

The word $\wt{w}$ obtained by reading $w$ from right to left is called the \emph{reversal} (or \emph{mirror image}) of $w$. A \emph{palindrome} is a word $w$ such that $\wt{w}=w$.
In particular, the empty word is  a palindrome. 

We recall the definitions of open and closed word given in \cite{Fi11}:

\begin{definition}\label{def:closed}
A word $w$ is \emph{closed} if and only if it is empty or has a factor $v\neq w$ occurring exactly twice in $w$, as a prefix and as a suffix of $w$ (with no internal occurrences). A word that is not closed is called \emph{open}.
\end{definition}

The word $aba$ is closed, since its factor $a$ appears only as a prefix and as a suffix. The word $abaa$, on the contrary, is not closed. Note that for any letter $a\in \Sigma$ and for any $n>0$, the word $a^{n}$ is closed, $a^{n-1}$ being a factor occurring only as a prefix and as a suffix in it (this includes the special case of single letters, for which $n=1$ and $a^{n-1}=\epsilon$). 

More generally, any word that is a power of a shorter word is closed. Indeed, suppose that $w=v^{n}$ for a non-empty $v$ and $n>1$. Without loss of generality, we can suppose that $v$ is not a power itself. If $v^{n-1}$ has an internal occurrence in $w$, then there exists a proper prefix $u$ of $v$ such that $uv=vu$, and it is a basic result in Combinatorics on Words that two words commute if and only if they are powers of a same shorter word, 
in contradiction with our hypothesis on $v$.

\begin{remark}
The notion of closed word is equivalent to that of \emph{periodic-like} word \cite{CaDel01a}. A word $w$ is periodic-like if its longest repeated prefix is not right special.

The notion of closed word is also closely related to the concept of \emph{complete return} to a factor, as considered in \cite{GlJuWiZa09}. A complete return to the factor $u$ in a word $w$ is any factor of $w$ having exactly two occurrences of $u$, one as a prefix and one as a suffix. Hence, $w$ is closed if and only if it is a complete return to one of its factors; such a factor is clearly both the longest repeated prefix and the longest repeated suffix of $w$ (i.e., the longest border of $w$). 
\end{remark}
 
\begin{remark}\label{obs}
Let $w$ be a non-empty word over $\Sigma$. The following characterizations of closed words follow easily from the definition:
 
\begin{enumerate}
 \item the longest repeated prefix (resp.~suffix) of $w$ does not have internal occurrences in $w$, i.e., occurs in $w$ only as a prefix and as a suffix;
  \item the longest repeated prefix (resp.~suffix) of $w$ is not a right (resp.~left) special factor of $w$;
 \item $w$ has a border that does not have internal occurrences in $w$;
 \item the longest border of $w$ does not have internal occurrences in $w$.
\end{enumerate}

Obviously, the negations of the previous properties characterizate open words. In the rest of the paper we will use these characterizations freely and without explicit mention to this remark.
\end{remark}

We conclude this section with two lemmas on right extensions.

\begin{lemma}\label{lem:nbo}
 Let $w$ be a non-empty word over $\Sigma$. Then there exists at most one letter $x\in\Sigma$ such that $wx$ is closed.
\end{lemma}

\begin{proof}
Suppose by contradiction that there exist $a,b\in\Sigma$ such
that both $wa$ and $wb$ are closed. Let $va$ and $v'b$ be the longest borders of $wa$ and $wb$, respectively. Since $va$ and $v'b$ are prefixes of $w$, one has that one is a prefix of the other. Suppose that $va$ is shorter than $v'b$. But then $va$ has an internal occurrence in $wa$ (that appearing as a prefix of the suffix $v'$) against the hypothesis that $wa$ is closed. \qed
\end{proof}

When $w$ is closed, then exactly one such extension is closed.
More precisely, we have the following (see also~\cite[Prop.~4]{CaDel01a}):

\begin{lemma}
 Let $w$ be a closed word. Then $wx$, $x\in \Sigma$, is closed if and only if $wx$ has the same period of $w$.
\end{lemma}

\begin{proof}
Let $w$ be a closed word and $v$ its longest border; in particular, $v$ is the longest repeated prefix of $w$. Let $x$ be the letter following the occurrence of $v$ as a prefix of $w$. Clearly, $wx$ is has the same period as $w$, and it is closed as its border $vx$ cannot have internal occurrences. Conversely, if $y\neq x$ is a letter, then $wy$ has a different period and it is open as its longest repeated prefix $v$ is right special.
 \qed
\end{proof}

For more details on open and closed words and related results see~\cite{CaDel01a,BuDelDel09,Fi11,FiLi12,BuDelFi13}.

\section{Open and Closed Prefixes of Sturmian Words}

 Let $\Sigma^{\omega}$ be the set of (right) infinite words over $\Sigma$, indexed by $\mathbb{N}_{0}$.  An element of $\Sigma^{\omega}$ is a \emph{Sturmian word} if it contains exactly $n+1$ distinct factors of length $n$, for every $n\ge 0$. 
 A famous example of Sturmian word is the Fibonacci word \[F=abaababaabaababaababa\cdots\] 

If $w$ is a Sturmian word, then $aw$ or $bw$ is also a Sturmian word. A Sturmian word $w$ is \emph{standard} (or \emph{characteristic}) if $aw$ and $bw$ are both Sturmian words. The Fibonacci word is an example of standard Sturmian word. In the next section, we will deal specifically with standard Sturmian words. Here, we focus on finite factors of Sturmian words, called \emph{finite Sturmian words}. Actually, finite Sturmian words are precisely the elements of $\Sigma^{*}$ verifying the following balance property:  for any $u,v\in \Fact(w)$ such that $|u|=|v|$ one has $||u|_{a}-|v|_{a}|\le 1$ (or, equivalently, $||u|_{b}-|v|_{b}|\le 1$). 

We let $\St$ denote the set of finite Sturmian words. The language $\St$ is factorial (i.e., if $w=uv\in \St$, then $u,v\in \St$) and extendible (i.e., for every $w\in \St$ there exist letters $x,y\in \Sigma$ such that $xwy\in \St$).

We recall the following definitions given in \cite{DelMi94}. 

\begin{definition}
A word  $w\in \Sigma^{*}$ is a  left special (resp.~right special) Sturmian word if $aw,bw\in \St$ (resp.~if $wa,wb\in \St$). A bispecial Sturmian word is a Sturmian word that is both left special and right special. 
\end{definition}

For example, the word $w=ab$ is a bispecial Sturmian word, since $aw$, $bw$, $wa$ and $wb$ are all Sturmian. This example also shows that a bispecial Sturmian word is not necessarily a bispecial factor of some Sturmian word (see \cite{Fi12} for more details on bispecial Sturmian words).

\begin{remark}\label{rem:rsp}
It is known that if $w$ is a left special Sturmian word, then $w$ is a prefix of a standard Sturmian word, and the left special factors of $w$ are prefixes of $w$. Symmetrically, if $w$ is a right special Sturmian word, then the right special factors of $w$ are suffixes of $w$.
\end{remark}


We now define the sequence of open and closed prefixes of a word.

\begin{definition}
 Let $w$ be a finite or infinite word over $\Sigma$. We define the sequence $\oc(w)$ as the sequence whose $n$-th element is $1$ if the prefix of length $n$ of $w$ is closed, or $0$ otherwise.
\end{definition}
For example, if $w=abaaab$, then $\oc(w)=101001$.

In this section, we prove the following:

\begin{theorem}\label{theor:main}
 Every (finite or infinite) Sturmian word $w$ is uniquely determined, up to isomorphisms of the alphabet $\Sigma$, by its sequence of open and closed prefixes $\oc(w)$.
\end{theorem}

We need some intermediate lemmas.

\begin{lemma}\label{lem:lsp}
Let $w$ be a right special Sturmian word and $u$ its longest repeated prefix. Then $u$ is a suffix of $w$.
\end{lemma}

\begin{proof}
If $w$ is closed, the claim follows from the definition of closed word. If $w$ is open, then $u$ is right special in $w$, and by
Remark~\ref{rem:rsp}, $u$ is a suffix of $w$.\qed
\end{proof}

\begin{lemma}\label{lem:speclo}
 Let $w$ be a right special Sturmian word.  Then $wa$ or $wb$ is closed.
\end{lemma}

\begin{proof}
Let $u$ be the longest repeated prefix of $w$ and $x$ the letter following the occurrence of $u$ as a prefix of $w$. By Lemma \ref{lem:lsp}, $u$ is a suffix of $w$. Clearly, the longest repeated prefix of $wx$ is $ux$, which is also a suffix of $wx$ and cannot have internal occurrences in $wx$ otherwise the longest repeated prefix of $w$ would not be $u$. Therefore, $wx$ is closed.
\qed
\end{proof}

So, by Lemmas \ref{lem:nbo} and \ref{lem:speclo}, if $w$ is a right special Sturmian word, then one of $wa$ and $wb$ is closed and the other is open. This implies that the sequence of open and closed prefixes of a (finite or infinite) Sturmian word characterizes it up to exchange of letters. The proof of Theorem \ref{theor:main} is therefore complete.

\section{Standard Sturmian Words}

In this section, we deal with the sequence of open and closed prefixes of standard Sturmian words. In \cite{BuDelFi13} a characterization of the sequence $\oc(F)$  of open and closed prefixes of the Fibonacci word $F$ was given.

Let us begin by recalling some definitions and basic results about standard Sturmian words. For more details, the reader can see \cite{Be07} or \cite{LothaireAlg}.

Let $\alpha$ be an irrational number such that $0<\alpha<1$, and let $\left[0;d_{0}+1,d_{1},\ldots\right]$ be the continued fraction expansion of $\alpha$.
The sequence of words defined by $s_{-1}=b$, $s_{0}=a$ and $s_{n+1}=s_{n}^{d_{n}}s_{n-1}$ for $n\ge 0$, converges to the infinite word $w_{\alpha}$, called the \emph{standard Sturmian word of slope $\alpha$}. 
The sequence of words $s_{n}$ is called the \emph{standard sequence} of $w_{\alpha}$. 

Note that $w_{\alpha}$ starts with letter $b$ if and only if $\alpha>1/2$, i.e., if and only if $d_{0}=0$. In this case, $\left[0;d_{1}+1,d_{2},\ldots\right]$ is the continued fraction expansion of $1-\alpha$, and $w_{1-\alpha}$ is the word obtained from $w_{\alpha}$ by exchanging $a$'s and $b$'s. Hence, without loss of generality, we will suppose in the rest of the paper that $w$ starts with letter $a$, i.e., that $d_{0}>0$. 

For every $n\ge -1$, one has 
\begin{equation}\label{eq:su}
s_{n}=u_{n}xy, 
\end{equation}
for $x,y$ letters such that $xy=ab$ if $n$ is odd or $ba$ if $n$ is even. Indeed, the sequence $(u_{n})_{n\ge -1}$ can be defined by: $u_{-1}=a^{-1}$, $u_0=b^{-1}$, and, for every $n\ge 1$, 
\begin{equation}\label{eq:un+1}
u_{n+1}=(u_{n}xy)^{d_{n}}u_{n-1}\,,
\end{equation}
where $x,y$ are as in \eqref{eq:su}.

\begin{example}
 The Fibonacci word $F$ is the standard Sturmian word of slope $(3-\sqrt{5})/2$, whose continued fraction expansion is $[0;2,1,1,1,\ldots]$, so that $d_{n}=1$ for every $n\ge 0$. Therefore, the standard sequence of the Fibonacci word $F$ is the sequence defined by: $f_{-1}=b$, $f_{0}=a$, $f_{n+1}=f_{n}f_{n-1}$ for $n\ge 0$. This sequence is also called the sequence of \emph{Fibonacci finite words}.
\end{example}

\begin{definition}
 A standard word is a finite word belonging to some standard sequence. A central word is a word $u\in\Sigma^*$ such that $uxy$ is a standard word, for letters $x,y\in \Sigma$. 
\end{definition}

It is known that every central word is a palindrome%
. Actually, central words play a central role in the combinatorics of Sturmian words and have several combinatorial characterizations (see \cite{Be07} for a survey). For example, a word over $\Sigma$ is central if and only if it is a palindromic bispecial Sturmian word.

\begin{remark}
Let $(s_{n})_{n\ge -1}$ be a standard sequence. It follows by the definition that for every $k\ge 0$ and $n\ge -1$, the word $s_{n+1}^{k}s_{n}$ is a standard word. In particular, for every $n\ge -1$, the word $s_{n+1}s_n=u_{n+1}yxu_nxy$ is a standard word. Therefore, for every $n\ge -1$, we have that 
\begin{equation}\label{eq:u}
u_nxyu_{n+1}=u_{n+1}yxu_n
\end{equation} is a central word.
\end{remark}

The following lemma is a well known result (cf.~\cite{fisch}). 

\begin{lemma}\label{lem:pref}
Let $w$ be a standard Sturmian word and $(s_{n})_{n\ge -1}$ its standard sequence. Then:
\begin{enumerate}
 \item
A standard word $v$ is a prefix of $w$ if and only if $v=s_{n}^{k}s_{n-1}$, for some $n\ge 0$ and $k\le d_{n}$.
 \item  A central word $u$ is a prefix of $w$ if and only if $u=(u_{n}xy)^{k}u_{n-1}$, for some $n\ge 0$, $0<k\le d_{n}$, 
 and distinct letters $x,y\in \Sigma$ such that $xy=ab$ if $n$ is odd or $ba$ if $n$ is even.
\end{enumerate}
\end{lemma}

Note that $(u_{n}xy)^{d_{n}+1}u_{n-1}$ is a central prefix of $w$, but 
this does not contradict the previous lemma since, by~\eqref{eq:un+1},
$(u_nxy)^{d_n+1}u_{n-1}=u_{n+1}yxu_n$.

Recall that a \emph{semicentral word} (see~\cite{BuDelFi13}) is a word in which the longest repeated prefix, the longest repeated suffix, the longest left special factor and the longest right special factor all coincide. It is known that a word $v$ is semicentral if and only if $v=uxyu$ for a central word $u$ and distinct letters $x,y\in \Sigma$. Moreover, $xuy$ is a factor of $uxyu$ and thus semicentral words are open, while central words are closed.

\begin{proposition}\label{prop:semi}
The semicentral prefixes of $w$ are precisely the words of the form $u_nxyu_n$, $n\ge 1$, where $x,y$ and $u_{n}$ are as in \eqref{eq:su}.
\end{proposition}

\begin{proof}
 Since $u_{n}$ is a central word, the word $u_nxyu_n$ is a semicentral word by definition, and it is a prefix of $u_nxyu_{n+1}=u_{n+1}yxu_{n}$, which in turn is a prefix of $w$ by Lemma \ref{lem:pref}.
 
 Conversely, assume that $w$ has a prefix of the form $u\xi\eta u$ for a central word $u$ and distinct letters $\xi,\eta\in \Sigma$. From Lemma \ref{lem:pref} and \eqref{eq:su}, we have that
 $$u\xi\eta u=(u_{n}xy)^{k}u_{n-1}\cdot \xi\eta \cdot (u_{n}xy)^{k}u_{n-1},$$
 for some $n\ge 1$, $k\le d_{n}$, and distinct letters $x,y\in \Sigma$ such that $xy=ab$ if $n$ is odd or $ba$ if $n$ is even. In particular, this implies that $\xi\eta=yx$.
 
If $k=d_{n}$, then $u=u_{n+1}yx u_{n+1}$, and we are done. So, suppose by contradiction that $k<d_{n}$. Now, on the one hand we have that $(u_{n}xy)^{k+1}u_{n-1}yx$ is a prefix of $w$ by Lemma \ref{lem:pref}, and so $(u_{n}xy)^{k+1}u_{n-1}$ is followed by $yx$ as a prefix of $w$; on the other hand we have
\begin{eqnarray*}
u\xi\eta u &=& (u_{n}xy)^{k}u_{n-1}\cdot yx \cdot (u_{n}xy)^{k}u_{n-1}  \\
&=& (u_{n}xy)^{k} \cdot u_{n-1}yxu_{n} xy \cdot(u_{n}xy)^{k-1}u_{n-1} \\
&=& (u_{n}xy)^{k} \cdot u_{n}xyu_{n-1} xy \cdot(u_{n}xy)^{k-1}u_{n-1} \\
&=& (u_{n}xy)^{k+1} \cdot u_{n-1}xy \cdot(u_{n}xy)^{k-1}u_{n-1},
\end{eqnarray*}
so that $(u_{n}xy)^{k+1}u_{n-1}$ is followed by $xy$ as a prefix of $w$, a contradiction.
 \qed
\end{proof}

The next theorem shows the behavior of the runs of open and closed prefixes in $w$ by determining the structure of the last elements of the runs.

\begin{theorem}
\label{thm:bdaries}
Let $vx$, $x\in \Sigma$, be a prefix of $w$. Then:
\begin{enumerate}
\item $v$ is open and $vx$ is closed if and only if there exists $n\ge 1$ such that $v=u_{n}xyu_{n}$;
\item $v$ is closed and $vx$ is open if and only if there exists $n\ge 0$ such that $v=u_{n}xyu_{n+1}=u_{n+1}yxu_{n}$.
\end{enumerate}
\end{theorem}

\begin{proof}
1. If $v=u_{n}xyu_{n+1}=u_{n+1}yxu_{n}$, then $v$ is semicentral and therefore open. The word $vx$ is closed since its longest repeated prefix $u_{n}x$ occurs only as a prefix and as a suffix in it.
 
Conversely, let $vx$ be a closed prefix of $w$ such that $v$ is open, and let $ux$ be the longest repeated suffix of $vx$. Since $vx$ is closed, $ux$ does not have internal occurrences in $vx$. Since $u$ is the longest repeated prefix of $v$ (suppose the longest repeated prefix of $v$ is a $z$ longer than $u$, then $vx$, which is a prefix of $z$, would be repeated in $v$ and hence in $vx$, contradiction) and $v$ is open, $u$ must have an internal occurrence in $v$ followed by a letter $y\neq x$. Symmetrically, if $\xi$ is the letter preceding the occurrence of $u$ as a suffix of $v$, since $u$ is the longest repeated suffix of $v$ one has that $u$ has an internal occurrence in $v$ preceded by a letter $\eta \neq \xi$. Thus $u$ is left and right special in $w$. Moreover, $u$ is the longest special factor in $v$. Indeed, if $u'$ is a left special factor of $v$, then $u$ must be a prefix of $u'$. But $ux$ cannot appear in $v$ since $vx$ is closed, and if $uy$ was a left special factor of $v$, it would be a prefix of $v$. Symmetrically, $u$ is  the longest right special factor in $v$. 
Thus $v$ is semicentral, and the claim follows from Proposition \ref{prop:semi}.
 
2. If $v=u_{n}xyu_{n+1}=u_{n+1}yxu_{n}$, then $v$ is a central word and therefore it is closed. Its longest repeated prefix is $u_{n+1}$. The longest repeated prefix of $vx$ is either $a^{d_0-1}$ (if $n=0$) or $u_{n}x$ (if $n>0$); in both cases, it has an internal occurrence as a prefix of the suffix $u_{n+1}x$. Therefore, $vx$ is open.
 
Conversely, suppose that $vx$ is any open prefix of $w$ such that $v$ is closed. If $vx=a^{d_0}b$, then $v=u_0xyu_1=u_1yxu_0$ and we are done. Otherwise, by 1), there exists $n\geq 1$ such that
$|u_n\xi y u_n|<|v|<|u_{n+1}y\xi u_{n+1}|$, where
$\{\xi,y\}=\{a,b\}$. We know that $u_n\xi y u_{n+1}$ is closed
and $u_n\xi y u_{n+1}\xi$ is open; it follows
$v=u_n\xi y u_{n+1}=u_nxyu_{n+1}$, as otherwise there should be in $w$ a semicentral prefix strictly between $u_nxyu_n$ and
$u_{n+1}yxu_{n+1}$.
 \qed
\end{proof}

Note that, for every $n\ge 1$, one has:
\begin{eqnarray*}
  u_{n+1}yxu_{n+1}&=&u_{n+1}yxu_{n}(u_{n}^{-1}u_{n+1})\\
 &=&u_{n}xyu_{n+1}(u_{n}^{-1}u_{n+1})\\
 &=&u_{n}xyu_{n}(u_{n}^{-1}u_{n+1})^{2}.
\end{eqnarray*}
Therefore, starting from an (open) semi-central prefix $u_{n}xyu_{n}$, one has a run of closed prefixes, up to the prefix $u_{n}xyu_{n+1}=u_{n+1}yxu_{n}=u_{n}xyu_{n}(u_{n}^{-1}u_{n+1})$, followed by a run of the same length of open prefixes, up to the prefix $u_{n+1}yxu_{n+1}=u_{n+1}yxu_{n}(u_{n}^{-1}u_{n+1})=u_{n}xyu_{n}(u_{n}^{-1}u_{n+1})^{2}$. See Table \ref{tab:example} for an illustration.

\begin{table}[ht]
\setlength{\tabcolsep}{10pt}
\begin{center}
\begin{tabular}{ l c l }
    prefix of $w$  &   open/closed & example  \\    \hline 
     $u_{n}xyu_{n}$           & open   & $aaba$        \\
     $u_{n}xyu_{n}x$          & closed   &    $aabaa$      \\   
     $u_{n}xyu_{n}xy$          & closed   & $aabaab$        \\
     \ldots      &   \ldots          & \ldots         \\
     $u_{n}xyu_{n+1}=u_{n+1}yxu_n$          & closed      & $aabaabaa$       \\
     $u_{n+1}yxu_{n}y$          & open   & $aabaabaaa$          \\
     $u_{n+1}yxu_{n}yx$           & open      & $aabaabaaab$       \\
	 \ldots      &   \ldots          & \ldots      \\
	 $u_{n+1}yxu_{n+1}$      & open   & $aabaabaaabaa$ \\
	 $u_{n+1}yxu_{n+1}y$        & closed   & $aabaabaaabaab$ \\
    \hline \\   
\end{tabular}
\end{center}
\caption{The structure of the prefixes of a standard Sturmian word $w=aabaabaaabaabaa\cdots$ with respect to the $u_{n}$ prefixes. Here $d_{0}=d_{1}=2$ and $d_{2}=1$.\label{tab:example}}
\end{table}

In Table \ref{tab:oc}, we show the first elements of the sequence $oc(w)$ for a standard Sturmian word $w=aabaabaaabaabaa\cdots$ of slope $\alpha=\left[0;3,2,1,\ldots\right]$, i.e., with $d_{0}=d_{1}=2$ and $d_{2}=1$. One can notice that the runs of closed prefixes are followed by runs of the same length of open prefixes. 

\begin{table}[ht]
\begin{small}
\begin{center}
\begin{tabular}{r*{15}{@{\hspace{1.4em}}c}}
 $n$    & 1 & 2 & 3 & 4 & 5 & 6 & 7 &
8 & 9 & 10 & 11 & 12 & 13 & 14 & 15 \\ \hline\\[-1ex]
 $w$    & $a$ & $a$ & $b$ & $a$ & $a$ & $b$ & $a$ &
$a$ & $a$ & $b$ & $a$ & $a$ & $b$ & $a$ & $a$ \\
\hline \\[-1ex]
$oc(w)$ & 1 & 1 & 0 & 0 & 1 & 1 & 1 & 1 & 0 & 0 & 0 & 0 & 1 & 1 & 1\\[0.5ex]
\hline \\
\end{tabular}
\end{center}
\end{small} 
\caption{\label{tab:oc}The sequence $oc(w)$ of open and closed prefixes for the word  $w= aabaabaaabaabaa\cdots$}
\end{table}

The words $u_{n}^{-1}u_{n+1}$ are reversals of standard words, for every $n\ge 1$. Indeed, let $r_n=\wt{s_n}$ for every $n\geq -1$, so that $r_{-1}=b$, $r_0=a$, and
$r_{n+1}=r_{n-1}r_n^{d_n}$ for $n\geq 0$. Since by \eqref{eq:su} $s_{n}=u_{n}xy$ and $s_{n+1}=u_{n+1}yx$, one has $r_n=yxu_{n}$ and $r_{n+1}=xyu_{n+1}$, and therefore, by \eqref{eq:u},
\begin{equation}\label{eq:ur}
 u_{n}r_{n+1}=u_{n+1}r_{n}.
\end{equation}
Multiplying \eqref{eq:ur}  on the left by $u_{n}^{-1}$ and on the right by $r_{n}^{-1}$, one obtains
\begin{equation}\label{eq:prop}
 r_{n+1}r_n^{-1}=u_n^{-1}u_{n+1}.
\end{equation}
Since $r_{n+1}=r_{n-1}r_{n}^{d_{n}}$, one has that $r_{n+1}r_{n}^{-1}=r_{n-1}r_{n}^{d_{n}-1}$, and therefore $r_{n+1}r_{n}^{-1}$ is the reversal of a standard word. By \eqref{eq:prop}, $u_n^{-1}u_{n+1}$ is the reversal of a standard word.

Now, note that for $n=0$, one has $u_{0}xyu_{1}=u_{1}yxu_{0}=a^{d_{0}}$ and $(u_{0}^{-1}u_{1})=ba^{d_{0}-1}$. Thus, we have the following:

\begin{theorem}\label{theor:decomp}
Let $w$ be the standard Sturmian word of slope $\alpha$, with $0<\alpha<1/2$, and let $[0;d_{0}+1,d_{1},\ldots]$, with $d_{0}>0$, be the continued fraction expansion of $\alpha$. The word $ba^{-1}w$ obtained from $w$ by swapping the first letter can be written as an infinite product of squares of reversed standard words in the following way:
$$ba^{-1}w=\prod_{n\ge 0}(u_{n}^{-1}u_{n+1})^{2},$$ 
where $(u_{n})_{n\ge -1}$ is the sequence defined in \eqref{eq:su}.

In other words, one can write
$$w=a^{d_{0}}ba^{d_{0}-1}\prod_{n\ge 1}(u_{n}^{-1}u_{n+1})^{2}.$$
\end{theorem}
 
\begin{example}
Take the Fibonacci word. Then, $u_{1}=\epsilon$, $u_{2}=a$, $u_{3}=aba$, $u_{4}=abaaba$, $u_{5}=abaababaaba$, etc. So, $u_{1}^{-1}u_{2}=a$, $u_{2}^{-1}u_{3}=ba$, $u_{3}^{-1}u_{4}=aba$, $u_{4}^{-1}u_{5}=baaba$, etc. Indeed, $u_{n}^{-1}u_{n+1}$ is the reversal of the Fibonacci finite word $f_{n-1}$. By Theorem \ref{theor:decomp}, we have:
\begin{eqnarray*}
 F &=& ab\prod_{n\ge 1}(u_{n}^{-1}u_{n+1})^{2}\\
 &=& ab\prod_{n\ge 0}(\wt{f_{n}})^{2}\\
 &=& ab\cdot (a\cdot a)(ba \cdot ba)(aba \cdot aba)(baaba\cdot baaba)\cdots
\end{eqnarray*}
i.e., $F$ can be obtained by concatenating $ab$ and the squares of the reversals of the Fibonacci finite words $f_n$ starting from $n=0$.

Note that $F$ can also be obtained by concatenating the reversals of the Fibonacci finite words $f_n$ starting from $n=0$:
\begin{eqnarray*}
 F &=& \prod_{n\ge 0} \wt{f_{n}}\\
 &=& a \cdot ba \cdot aba \cdot baaba \cdot ababaaba \cdots
\end{eqnarray*}
and also by concatenating $ab$ and the Fibonacci finite words $f_n$ starting from $n=0$:
\begin{eqnarray*}
 F &=& ab\prod_{n\ge 0} f_{n}\\
 &=& ab \cdot a \cdot ab \cdot aba \cdot abaab \cdot abaababa \cdots
\end{eqnarray*}
\end{example}

One can also characterize the sequence of open and closed prefixes of a standard Sturmian word $w$ in terms of the  directive sequence of $w$.

Recall that the \emph{continuants} of an integer sequence $(a_n)_{n\geq 0}$ are defined as
$K\left[\ \, \right]=1$, $K\left[a_0\right]=a_0$, and, for every $n\geq 1$,
\[K\left[a_0,\ldots,a_n\right]=a_nK\left[a_0,\ldots,a_{n-1}\right]+K\left[a_0,\ldots,a_{n-2}\right].\] 
Continuants are related to continued fractions, as the $n$-th convergent of
$[a_0; a_1,a_2,\ldots]$ is equal to $K\left[a_0,\ldots, a_n\right]/K\left[a_1,\ldots, a_n\right]$.

Let $w$ be a standard Sturmian word and $(s_n)_{n\ge -1}$ its standard sequence. Since $|s_{-1}|=|s_0|=1$ and, for every $n\ge 1$,
$|s_{n+1}|=d_n|s_n|+|s_{n-1}|,$ then one has, by definition, that for every $n\geq 0$
\[|s_n|=K\left[1,d_0,\ldots, d_{n-1}\right].\]

For more details on the relationships between continuants and Sturmian words see \cite{dL13}.

By Theorems~\ref{thm:bdaries} and~\ref{theor:decomp}, all prefixes up to
$a^{d_0}$ are closed; then all prefixes from $a^{d_0}b$ till $a^{d_0}ba^{d_0-1}$ are open,
then closed up to $a^{d_0}ba^{d_0-1}\cdot u_1^{-1}u_2$, open again up to 
$a^{d_0}ba_{d_0-1}\cdot (u_1^{-1}u_2)^2$, and so on.
Thus, the lengths of the successive runs of closed and open prefixes are: 
$d_0$, $d_0$, $|u_2|-|u_1|$, $|u_2|-|u_1|$, $|u_3|-|u_2|$, $|u_3|-|u_2|$, etc. 
Since $d_0=K\left[1,d_0-1\right]$ and, for every $n\geq 1$,
\[\begin{split}
 |u_{n+1}|-|u_n|&= |s_{n+1}|-|s_n|=(d_n-1)|s_n|+|s_{n-1}|\\
 &=K\left[1,d_0,\ldots,d_{n-1},d_n-1\right],
\end{split}\]
we have the following:

\begin{corollary}\label{cor:formula}
Let $w$ and $\alpha$ be as in the previous theorem and let, for every $n\geq 0$, $k_n=K\left[1,d_0,\ldots,d_{n-1},d_n-1\right]$. Then
\[oc(w)=\prod_{n\geq 0}1^{k_n}0^{k_n}.\]
\end{corollary}

\section*{Acknowledgments}

We thank an anonymous referee for helpful comments that led us to add the formula in Corollary \ref{cor:formula} to this final version. We also acknowledge the support of the PRIN 2010/2011 project ``Automi e Linguaggi Formali: Aspetti Matematici e Applicativi'' of the Italian Ministry of Education (MIUR).

\bibliographystyle{splncs}

\begin{thebibliography}{10}

\bibitem{BuDelFi13}
Bucci, M., De~Luca, A., Fici, G.:
\newblock Enumeration and {S}tructure of {T}rapezoidal {W}ords.
\newblock Theoretical Computer Science \textbf{468} (2013)  12--22

\bibitem{CaDel01a}
Carpi, A., de~Luca, A.:
\newblock Periodic-like words, periodicity and boxes.
\newblock Acta Informatica \textbf{37} (2001)  597--618

\bibitem{LothaireAlg}
Lothaire, M.:
\newblock Algebraic Combinatorics on Words.
\newblock Encyclopedia of Mathematics and its Applications. Cambridge Univ.
  Press, New York, NY, USA (2002)

\bibitem{Fi11}
Fici, G.:
\newblock A {C}lassification of {T}rapezoidal {W}ords.
\newblock In: {WORDS}~2011, 8th International Conference on Words. Volume~63 of
  Electronic Proceedings in Theoretical Computer Science. (2011)  129--137

\bibitem{GlJuWiZa09}
Glen, A., Justin, J., Widmer, S., Zamboni, L.Q.:
\newblock Palindromic richness.
\newblock European J. Combin. \textbf{30} (2009)  510--531

\bibitem{BuDelDel09}
Bucci, M., de~Luca, A., De~Luca, A.:
\newblock Rich and {P}eriodic-{L}ike {W}ords.
\newblock In: {DLT}~2009, 13th International Conference on Developments in
  Language Theory. Volume 5583 of Lecture Notes in Comput. Sci.
\newblock Springer (2009)  145--155

\bibitem{FiLi12}
Fici, G., Lipt{\'a}k, {\relax Zs}.:
\newblock Words with the {S}mallest {N}umber of {C}losed {F}actors.
\newblock In: 14th Mons Days of Theoretical Computer Science. (2012)

\bibitem{DelMi94}
de~Luca, A., Mignosi, F.:
\newblock Some combinatorial properties of {S}turmian words.
\newblock Theoret. Comput. Sci. \textbf{136} (1994)  361--385

\bibitem{Fi12}
Fici, G.:
\newblock A {C}haracterization of {B}ispecial {S}turmian {W}ords.
\newblock In: {MFCS}~2012, 37th International Symposium on Mathematical
  Foundations of Computer Science. Volume 7464 of Lecture Notes in Comput.
  Sci., Springer Berlin Heidelberg (2012)  383--394

\bibitem{Be07}
Berstel, J.:
\newblock Sturmian and episturmian words.
\newblock In: {CAI}~2007, Second International Conference on Algebraic
  Informatics. Volume 4728 of Lecture Notes in Computer Science., Springer
  (2007)  23--47

\bibitem{fisch}
Fischler, S.:
\newblock Palindromic prefixes and episturmian words.
\newblock J. Combin. Theory Ser. A \textbf{113} (2006)  1281--1304

\bibitem{dL13}
de~Luca, A.:
\newblock Some extremal properties of the {F}ibonacci word.
\newblock Internat. J. Algebra Comput. (to appear)

\end{thebibliography}

\end{document}